\documentclass[10pt,letterpaper]{amsart}
\usepackage{amssymb,graphicx}
\usepackage[T1]{fontenc} \usepackage{mathptmx} 
\newcommand{\indel}[1]{\partial/\partial #1}

\newtheorem{theorem}{Theorem}%[section]
\newtheorem{lemma}[theorem]{Lemma}
\newtheorem{proposition}[theorem]{Proposition}

\theoremstyle{remark}
\newtheorem{remark}[theorem]{Remark}

\newtheorem{definition}[theorem]{Definition}

 \def\RR{{\mathbb R}}

\def\NN{{\mathbb N}}

\title{Quasihomogeneous three-dimensional real analytic Lorentz metrics}

\author{Sorin Dumitrescu}\address{Universit\'e Nice-Sophia Antipolis, Laboratoire J.-A. Dieudonn\'e, UMR 7351 CNRS, Parc Valrose, 06108 Nice Cedex 2, France} \email{dumitres@unice.fr}


\begin{document}
\begin{abstract} We classify germs  at the origin of   real analytic Lorentz metrics  on $\mathbf{R}^3$  which are {\it quasihomogeneous}, in the sense that they are {\it locally homogeneous}  on an open set containing the origin in its closure, but not locally homogeneous in the neighborhood of the origin.
\end{abstract}

\maketitle

\section{Introduction}

The most symmetric Riemannian and pseudo-Riemannian metrics are those for which the {\it Killing Lie algebra} is transitive. They are called {\it locally homogeneous} and their study is a traditional field in differential and Riemannian geometry. In dimension two, this means exactly that the (pseudo)-Riemannian  metric is of constant sectional curvature. Locally homogeneous  Riemannian metrics of dimension three, are the context of Thurston's 3-dimensional geometrization program~\cite{thurston}. The classification of compact locally homogeneous  Lorentz $3$-manifolds was given in~\cite{dumzeg}.

  This article deals with the classification of  germs at the origin of  three-dimensional  real analytic Lorentz metrics   which are~{\it quasihomogeneous}, in the sense that they are {\it locally homogeneous}  on an open set containing the origin in its closure, but not locally homogeneous in the neighborhood of the origin. The quasihomogeneous  Lorentz metrics  are the most symmetric ones, after those which are
  locally homogeneous. In particular, all their scalar invariants are constant. Recall that, in the Riemannian setting, that implies local homogeneity (see~\cite{Tri} for an effective result).

The main theorem in this article  is the following:

\begin{theorem} \label{main} Let~$g$ be a   real-analytic Lorentz metric  in a neighborhood of the origin in~$\mathbf{R}^3$. Suppose that the maximal  open set  on which  the Killing  Lie algebra  $\mathcal{G}$   of~$g$ is transitive contains the origin in its closure,  but does not contain the origin. Then, in adapted analytic coordinates in the neighborhood of the origin,  the germ of~$g$ at~$0$ is 
$$ g=dx^2+ dhdz+Cz^2dh^2$$ with $C \in \RR \setminus \{ 0 \}$.

In these coordinates  $\mathcal{G}=\langle \indel{x}, \indel{h},h\indel{h}-z\indel{z}\rangle.$
In particular, the Killing Lie algebra is isomorphic  to $\RR \oplus \mathfrak{aff}(\mathbf{R})$,  where $\mathfrak{aff}(\mathbf{R})$ is  the Lie algebra of the affine group of the real line. 
\end{theorem}

\begin{remark}
 \emph{The closed set on which $g$ is not locally homogeneous is the totally geodesic surface $S$ given by  $z=0$}. 
 The curvature tensor of $g$ vanishes exactly on $S$.
 The  Killing Lie algebra $\mathcal{G}$ also preserves the flat Lorentz metric $g_0=dx^2+ dhdz$.
\end{remark}

The quasihomogeneous germs of Lorentz metrics constructed in theorem~\ref{main} never extend to a compact manifold. Indeed, the main theorem in~\cite{Dumitrescu} asserts  that {\it a real analytic Lorentz metric on a compact $3$-manifold which is locally homogeneous on a nontrivial open set is locally homogeneous on all of the manifold}. The same is known to be true  for real analytic Lorentz  metrics in higher dimension, under the assumptions that the Killing algebra is semi-simple, the metric is geodesically complete  and the universal cover is acyclic, as a consequence of a more general result of K. Melnick~\cite{Melnick}.

In the smooth category, A. Zeghib proved in~\cite{Zeghib} that compact Lorentz $3$-manifolds which admit essential Killing fields are necessarily locally homogeneous.

These  results are  motivated by Gromov's \emph{open-dense orbit theorem}~\cite{DG,Gro} (see also~\cite{Benoist2,Feres}). Gromov's result asserts that, if the pseudogroup of local automorphisms of  a \emph{rigid geometric structure} (a Lorentz metric or an analytic connection, for example) acts with a dense orbit, then this orbit is open.  In this case, the rigid geometric structure is locally homogeneous on an open dense set. Gromov's theorem says nothing about this maximal open and dense set of local homogeneity which appears to be mysterious. In many interesting geometric situations, it may be all of the (connected)  manifold. This was proved, for instance,  for Anosov flows with differentiable stable and unstable foliations and  transverse contact  structure~\cite{BFL}. In~\cite{BF}, the authors deal with this question  and their results indicate ways in which some rigid geometric structures cannot degenerate off the open dense set. 

In a recent common work with A. Guillot we obtained, with different methods,  the  classification of germs of quasihomogeneous real analytic torsion free affine connections on surfaces~\cite{dumiguillot}.

The composition of this article is the following. In Section~\ref{section2} we use the geometry of Killing fields and geometric  invariant theory to prove that the Killing Lie algebra  of a  three-dimensional quasihomogeneous Lorentz metric $g$  is a three-dimensional solvable Lie algebra. We also show that $g$ is locally homogeneous away of a (totally geodesic) surface $S$ on which the isotropy is an one
parameter semi-simple group. Theorem~\ref{main} is proved in Section~\ref{section3}.
As a by-product of the proof  of Theorem~\ref{main}, we get  the following more technical result.

\begin{proposition} \label{second}  Let~$g$ be a   real-analytic Lorentz metric  in a neighborhood of the origin in~$\mathbf{R}^3$. Suppose  that there exists a three-dimensional solvable  subalgebra  of the Killing Lie algebra acting transitively on an
open set admitting the origin in its closure, but not in the neighborhood of the origin.  If the isotropy at the origin is an one parameter semi-simple subgroup in $O(2,1)$, then there exists  local analytic coordinates $(x,h,z)$ in the neighborhood of the origin, and real constants $C,D$ such that  the germ of $g$ at $0$ is 
$$dx^2+ dhdz+Cz^2dh^2+Dzdxdh.$$

In these coordinates, the three-dimensional solvable subalgebra is  $\mathcal{G}=\langle \indel{x}, \indel{h},h\indel{h}-z\indel{z}\rangle.$
In particular, it is isomorphic  to $\RR \oplus \mathfrak{aff}(\mathbf{R})$,  where $\mathfrak{aff}(\mathbf{R})$ is  the Lie algebra of the affine group of the real line. \\

 (i) The Killing algebra is of dimension three (and $g$ is quasihomogeneous) if and only if $C \neq 0$ and $D =0$.\\

(ii) If $C=0$ and $D \neq 0$, the Killing algebra is solvable 
of dimension $4$ and $g$ is locally homogeneous. An extra Killing field is $e^{-Dx} \indel{z}$. \\

(iii) If $C \neq 0$ and $D \neq 0$, then $g$ is isomorphic to a left invariant metric on $SL(2, \RR)$ and the Killing algebra is $\RR \oplus sl(2, \RR)$. An  extra Killling field is  $T=ah \indel{x}+ \frac{1}{2}bh^2 \indel{h} +(-bzh- \frac{a}{D}) \indel{z}$, with $a,b$ real constants such that $a(D - \frac{C}{D})=b$.\\

(iv) If $C=0$ and $D=0$, $g$ is flat and the Killing  algebra is of dimension six.\\
\end{proposition}

\section{Killing Lie Algebra. Invariant Theory}   \label{section2}

Let  $g$ be a real analytic  Lorentz metric defined in an open  neighborhood $U$ of the origin in $\RR^3$.

Classically  (see, for instance~\cite{Gro,DG}) one  consider the $k$-jet of $g$, by taking at each point $u \in U$ the expression of $g$ in exponential coordinates, up to order $k$. In these coordinates, the $1$-jet of $g$ is
the standard flat Lorentz metric $dx^2+dy^2-dz^2$. At each point $u \in U$, the space of exponential coordinates is acted on simply transitively by $O(2,1)$ (which is  isomorphic to $PSL(2, \RR)$).
The space of all exponential coordinates in $U$ is a principal $PSL(2,\RR)$-bundle over $U$, also called the orthonormal frame bundle and denoted by  $R(U)$.

Geometrically, the $k$-jet of $g$ is an (analytic) $PSL(2,\RR)$-equivariant map $g^{(k)}: R(U) \to V^{(k)}$, where  $V^{(k)}$ is the  finite dimensional  affine space of $k$-jets of Lorentz metrics with origin at the $1$-jet $dx^2+dy^2-dz^2$, endowed with the linear action of $O(2,1) \simeq PSL(2, \RR)$ (notice that this action preserves the origin). More precisely, consider a system of  local exponential coordinates at $u$ with respect to $g$ and, for all $k \in \NN$  take the $k$-jet of $g$ in these  coordinates. Any  linear isomorphism of $(T_{u}U,g(u))$ gives another system of local exponential coordinates at $u$, with respect to which we consider the $k$-jet of $g$. This gives a linear  {\it algebraic}  $PSL(2,\RR)$-action on the vector space $V^{(k)}$ of $k$-jets of Lorentz metrics in exponential coordinates. One can find the details of this classical construction in~\cite{DG}.\\

Recall also that a (local) vector field is a  {\it Killing field} for a Lorentz metric $g$ if its (local) flow preserves $g$. The collection of all local Killing fields in the neighborhood of a point has the structure of a finite dimensional Lie algebra called the {\it Killing algebra} of $g$. We will denote it by $\mathcal G$.

At  a given point $u \in U$, the subalgebra $\mathcal I$ of $\mathcal G$ consisting on those vector fields $X \in \mathcal G$  such that $X(u)=0$,  is called the {\it isotropy} algebra at $u$.

\begin{definition} The  Lorentz metric  $g$ is said to be locally homogeneous on an open subset $W \subset U$, if for any $w \in W$ and any  tangent vector $V  \in T_{w}W$, there exists a local Killing field $X$ of
$g$ such that $X(w)=V.$ In this case, we will say that the Killing algebra $\mathcal G$ is transitive on $W$.
\end{definition}

 Our proof   of theorem~\ref{main}  will need analyticity in an essential way. We will make use of an extendability  result  for  local  Killing fields proved first by Nomizu in the Riemannian setting~\cite{Nomizu} and generalized then for  any rigid geometric structures by Amores and Gromov~\cite{Amores,Gro,DG}. This phenomenon states  that a  local Killing field of $g$ can  be extended  by monodromy along any curve $\gamma$  in $U$ and the resulting Killing field only depends on  the homotopy type of $\gamma$.
 
 Here we will assume that $U$ is  connected and simply connected  and hence, {\it  local Killing fields extend on all of $U$. }
  In particular, the Killing  algebra in the neighborhood of any point is the same.

 Notice  that Nomizu's extension phenomenon doesn't imply that the extension of a family of  pointwise linearly independent Killing fields, stays linearly independent.\\
 
{\it Assume for now on  that  $g$ is a   quasihomogeneous Lorentz metric in the neighborhood of the origin in $\RR^3$.}\\

  The set of points $s$  in $U$ at which  the Killing algebra  $\mathcal G$ of $g$ does  not span the  tangent space $T_sU$  is a nontrivial locally closed  analytic subset $S$ in $U$ passing through 
origin. In this case, $g$  is locally homogeneous on each connected component of $U \setminus S$, but not in the neighborhood of points in $S$.
 
 Moreover, at  each point $s$ of $S$ one will get a {\it nontrivial isotropy algebra}  given by the kernel of the canonical evaluation morphism $ev : \mathcal G \to T_sU$ (which is, by definition, of nonmaximal rank  at points of $S$).
 
 Let us prove  the following 
 
 \begin{lemma} \label{3}The Killing algebra $\mathcal G$ cannot be both  3-dimensional and unimodular.
 \end{lemma}
 
 \begin{proof} Let $K_1$, $K_2$ and $K_3$ be a basis of the Killing algebra. Consider the analytic function $v=vol(K_1, K_2, K_3)$, where $vol$ is the volume of the Lorentz metric. Since $\mathcal G$ is unimodular, the function $v$ is $\neq 0$ and  constant on each open set where $\mathcal G$ is transitive. On the other hand, $v$ vanishes  on $S$: a contradiction.
 \end{proof}

 We prove now that:

\begin{lemma} \label{dim}  (i) The dimension of the isotropy  at a point $u \in U$ is $\neq 2$.

                             (ii) The Killing algebra $\mathcal G$  is of dimension $3$.
                             
                             (iii) The Killing algebra  $\mathcal G$ is solvable.
\end{lemma}

\begin{proof} (i) Assume by contradiction that the isotropy algebra $\mathcal{I}$ at a point $u \in U$ has dimension two. 
 
 Elements of $\mathcal{I}$ linearize in exponential coordinates at $u$.

Since elements of $\mathcal{I}$  preserve $g$, they preserve, in particular,  the $k$-jet of $g$ at $u$, for all $k \in \NN$.
 This gives an embedding of $\mathcal{I}$ in the Lie algebra of $PSL(2,\RR)$ such that  the corresponding (two dimensional)   connected subgroup of $PSL(2,\RR)$ preserves the $k$-jet of $g$ at $u$, for all  $k \in \NN$. 
 
 Now we use the  fact that {\it the stabilizers of a  finite dimensional  linear  algebraic $PSL(2,\RR)$-action are of dimension  $\neq 2$}. Indeed, it suffices to check this statement for irreducible
 linear representations of $PSL(2, \RR)$, for which it is well-known that the stabilizer in $PSL(2,\RR)$  of a nonzero element is one dimensional~\cite{Kir}.
 
 It follows that the stabilizer in $PSL(2,\RR)$  of the $k$-jet of $g$ at $u$ is of dimension three and hence contains the connected component of identity in $PSL(2, \RR)$. Consequently,  in exponential coordinates  at $u$, each element of the connected  component of the identity in $PSL(2, \RR)$  gives rise  to a local linear vector field which preserves $g$ (for it preserves all $k$-jets
 of $g$). The isotropy algebra $\mathcal{I}$ contains a copy of the Lie algebra of $PSL(2, \RR)$: a contradiction, since $\mathcal{I}$ is of dimension two.
 
 (ii) Since $g$ is quasihomogeneous, the Killing algebra is of dimension at least $3$.
 
 For a three-dimensional Lorentz metric, the maximal dimension of the Killing algebra is $6$. This characterizes Lorentz metrics of constant sectional curvature. Indeed, in this case,  the isotropy is, at each point,  of dimension three and acts transitively on the nondegenerate  2-plans (see, for instance, ~\cite{Wolf}). These Lorentz metrics are locally homogeneous (and not quasihomogeneous).
 
 Assume, by contradiction, that the Killing algebra of $g$  is of dimension $5$. Then, on any open set of local homogeneity the isotropy is two-dimensional. This is in contradiction with point (i).
 
 Assume by contradiction that the Killing algebra of $g$  is of dimension $4$. Then, at a point $s \in S$,  the isotropy has dimension $\geq 2$. Hence, point (i) implies that the isotropy  at  $s$  has dimension three (isomorphic to $PSL(2, \RR))$. Moreover, the  standard linear action of the isotropy on $T_sU$ preserves the image of the evaluation morphism $ev(s): \mathcal G \to T_sU$, which is a line. But the standard $3$-dimensional  $PSL(2, \RR)$-representation does not admit invariant lines: a contradiction. Therefore, the Killing algebra is three-dimensional.
 
 (iii) A Lie algebra of dimension three is semi-simple or solvable~\cite{Kir}. Since semi-simple Lie algebras are unimodular, Lemma~\ref{3} implies that  $\mathcal G$ is solvable.

 \end{proof}

\begin{lemma} \label{unimodular}   $S$ is a connected  real analytic submanifold of codimension one.

  \end{lemma}
  
\begin{proof}

 If needed, one can  shrink the open set $U$ in order to get $S$ connected. Let $(K_1, K_2, K_3)$ a basis of the Killing algebra $\mathcal G$.  Then $S$ is defined by the equation $v=vol(K_1,K_2,K_3)=0$, which is an analytic subset  in $U$.  By  point (i) in Lemma~\ref{dim}, the isotropy algebra at points in  $S$ has dimension one or three. We prove that this dimension must  be equal to one.

Assume, by contradiction, that there exists $s \in S$ such that the isotropy at $s$ has dimension three. Then, the isotropy algebra at $s$ is isomorphic to the Lie algebra of (the full) linear group $PSL(2, \RR)$.
On the other hand, since both are $3$-dimensional,  the isotropy algebra at $s$ is isomorphic to $\mathcal G$. Hence, $\mathcal G$ is isomorphic to the Lie algebra of $PSL(2, \RR)$ which is semi-simple.
This is in contradiction with Lemma~\ref{dim} (point (iii)).

It follows that the isotropy algebra at each point $s \in S$ is of dimension one. Equivalently, the evaluation morphism $ev(s): \mathcal G \to T_sU$ has rank two. Since the $\mathcal G$-action preserves
$S$, this implies that $S$ is a smooth submanifold of codimension one in $U$ and $T_sS$ coincides with the image of $ev(s)$. Moreover, $\mathcal G$ acts transitively on $S$ (which locally coincides
with the $\mathcal G$- orbit  of  $s$ in $U$).

\end{proof}

Let us   recall Singer's result~\cite{Singer,DG,Gro} which asserts that {\it $g$ is locally homogeneous if and only if the image of $g^{(k)}$ is exactly one $PSL(2, \RR)$-orbit in $V^{(k)}$, for a certain $k$ (big enough).}

       As a consequence of Singer's theorem we get:

\begin{proposition} \label{invariant theory}  If $g$ is quasihomogeneous, then the Killing algebra $\mathcal G$ does not preserve any vector field of constant norm $\leq 0$.
\end{proposition}

\begin{proof} Let $k \in \NN$ be given by Singer's theorem.

First suppose, by contradiction, that  there exists a  isotropic vector field $X$ in $U$,  preserved by $\mathcal G$. Then the $\mathcal G$-action  on $R(U)$ (lifted from the action on $U$) preserves the subbundle $R'(U)$, which is a reduction 
of the structural group $PSL(2, \RR)$ to the stabilizer $H=\left(\begin{array}{cc} 1 & T \\ 0 & 1\\ \end{array}\right)$ (with  $T \in \RR$) of an isotropic vector in the standard linear representation of 
$PSL(2, \RR) \simeq O(2,1)$ on $\RR^3$.  One consider now only exponential coordinates with respect to frames preserving  $X$ and get a  $H$-equivariant $k$-jet map $g^{(k)}:R'(U) \to V^{(k)}$.

On each open set $W$ on which  $g$ is locally homogeneous, the image  $g^{(k)}(R'(W))$ is exactly one $H$-orbit $\mathcal O$  in $V^{(k)}$. Let $s \in S$ being a point in the closure of $W$. Then the image
through $g^{(k)}$ of the fiber of $R'(W)_s$ above $s$ lies in the closure of $\mathcal O$. But, here  $H$ is unipotent  and a classical result due to Konstant and Rosenlicht~\cite{Rosenlicht} asserts that 
{\it for algebraic representations of unipotent groups, the orbits are closed. } This implies that the image $g^{(k)}(R'(W)_s)$ is also $\mathcal O$. Since $\mathcal G$ acts transitively on $S$, this holds for all
$s \in S$.

But any open set of local homogeneity in $U$ admits points of $S$ in its closure. It follows that the image of $R'(U)$ through $g^{(k)}$ is exaclty the orbit $\mathcal O$ and Singer's  theorem
implies that $g$ is locally homogeneous (and not quasihomogeneous): a contradiction.

If, also by contradiction, there exists  a $\mathcal G$-invariant    vector field $X$ in $U$,    of constant strictly negative $g$-norm, then the  $\mathcal G$-action on $R(U)$  also preserves a subbundle $R'(U)$, with structural group $H'$. Here
$H'$  is the stabilizer  of a strictly negative vector  in the standard linear representation of 
$PSL(2, \RR)$ on $\RR^3$. Moreover, $H'$  is  {\it a compact one parameter (elliptic)  group} in $PSL(2, \RR)$. The previous argument works, replacing  Konstant-Rosenlicht theorem, by the obvious  fact that  orbits of compact group (smooth) actions are closed.

\end{proof}

\begin{lemma}  \label{X} At each point $s$  in $S$, the isotropy  is   an  one parameter semi-simple subgroup in $PSL(2, \RR)$. 
  \end{lemma}
  
  \begin{proof} Pick up a point $s \in S$ and consider a vector $X(s) \in T_{s}U$ which is fixed by the isotropy at $s$. Three distinct possibilities might appear: either $X(s)$ is isotropic 
  (the isotropy corresponds to an one parameter unipotent subgroup in $PSL(2,\RR)$), or $X(s)$ is of strictly negative norm (the isotropy corresponds to an   one parameter elliptic subgroup in $PSL(2,\RR)$),
  or $X(s)$ 
  is of strictly  positive norm (the isotropy corresponds to an one parameter  semi-simple subgroup in $PSL(2,\RR)$).

  Consider a real analytic arc  $c(t)$ transverse to $S$ in $s$ and extend
  $X(s)$ to  a real analytic vector field $\bar{X}$ of constant $g$-norm  defined along the curve  $c(t)$ (this extension is not unique).  Then extend  $\bar{X}$ in a unique way, by $\mathcal G$-invariance,  
  to  a  real analytic vector field $X$,  defined  in a neighborhood of $s$ in $U$ . The vector field $X$ is well defined  even on $S$ (on which the $\mathcal G$-action is not simply transitive), because $X(s)$ is invariant by the isotropy.
  
  By Proposition~\ref{invariant theory}, the $\mathcal G$-invariant vector field $X$ should be of constant strictly positive norm. Hence, the isotropy at $s$  is an one parameter semi-simple subgroup in $PSL(2,\RR)$.
  \end{proof}

  \section{Quasihomogeneous Lorentz metrics with semi-simple isotropy}   \label{section3}

 In this section we  settle the remaining  case, where the isotropy  at $S$ is semi-simple.  Lemma~\ref{X}  constructed   a  $\mathcal G$-invariant  vector field  $X$, in $U$, of constant strictly  positive $g$-norm.
We  normalize $X$ and suppose that $X$ is of  constant $g$-norm equal  to $1$. This vector field is not unique (only its restriction to $S$ is) and  the following Lemma~\ref{Killing} shows that among all 
$\mathcal G$-invariant vector fields of norm $1$, there exists  exactly one (denoted by $X'$) which is Killing.

In the sequel, we will work only with  the restriction of $X$ to $S$, which we  will still denote  by $X$.

  Recall  that the affine group of the real  line  $Aff$  is the group of transformations of $\RR$, given by $x \to ax+b$, with $a \in \RR^*$ and $b \in \RR$. If $Y$ is the infinitesimal generator of the one parameter group of  homotheties and $H$ the infinitesimal generator of the one parameter group of  translations,  then $\lbrack Y, H \rbrack =H$.

\begin{lemma}  \label{Killing}  (i) The Killing algebra  $\mathcal G$ is isomorphic to that of $\RR \times Aff$. The isotropy corresponds to  the one parameter group of homotheties in $Aff$.

(ii) The vector field $X$ is the restriction to $S$ of a central element $X'$ in $\mathcal G$.

(iii)   The restriction of the Killing algebra to $S$ has, in adapted analytic coordinates $(x,h)$,  the following basis $(-h \frac{\partial}{\partial h}, \frac{\partial}{\partial h},\frac{\partial}{\partial x})$.

(iv) In the previous coordinates, the restriction of $g$ to $S$ is $dx^2$.

\end{lemma}

\begin{proof}

 (i)  We show  first that the derivative Lie algebra $\lbrack \mathcal{G}, \mathcal{G}  \rbrack$ is 1-dimensional.

 It is a general fact that the derivative algebra of a solvable Lie algebra is nilpotent~\cite{Kir}. Remark first that $\lbrack \mathcal G, \mathcal G \rbrack \neq 0$. Indeed,   if not $\mathcal G$ is abelian and the action of the isotropy $\mathcal I \subset \mathcal G$ at a point $s \in S$ would be trivial
on $\mathcal G$ and hence on $T_{s}S$, which is identified to $\mathcal G / \mathcal I$. The isotropy action  on the tangent space $T_sS$ being trivial,
this  implies that the isotropy action is trivial on $T_{s} U$ (an element of $O(2,1)$ which acts trivially on a plane in $\RR^3$ is trivial). This implies that the isotropy is trivial at $s \in S$: a contradiction.

As $\mathcal G$ is 3-dimensional, its derivative algebra $\lbrack \mathcal G, \mathcal G \rbrack$
is a nilpotent Lie algebra of dimension $1$ or $2$, hence $\lbrack \mathcal G, \mathcal G \rbrack \simeq \RR$, or $\lbrack \mathcal G, \mathcal G \rbrack \simeq \RR^2$.

Assume, by contradiction, that $\lbrack \mathcal G, \mathcal 
G \rbrack \simeq \RR^2$. \\

  We first  prove  that the isotropy $\mathcal I$ lies in $\lbrack \mathcal G, \mathcal G \rbrack$. Assume, by contradiction, that this is not the case. Then,  $\lbrack \mathcal G, \mathcal G \rbrack \simeq \RR^2$ will act 
  freely and so transitively on $S$, preserving the vector field $X$.  In particular, $X$ is the restriction to $S$ of a Killing vector  field $X' \in \lbrack \mathcal G, \mathcal G \rbrack$.  

Let $Y$ be a generator of the isotropy at $s \in S$. Since $X$ is fixed by the isotropy, one gets, in restriction to $S$,  the following Lie bracket relation: $\lbrack Y, X' \rbrack = \lbrack Y,X \rbrack=aY$,
for some $a \in \RR$. On the other hand, $Y$ is supposed  not  to belong to $\lbrack \mathcal G, \mathcal G \rbrack$, meaning that $a=0$.

This  implies that $X'$ is a central element in $\mathcal G$. In particular, $\lbrack \mathcal G, \mathcal G \rbrack$
is one-dimensional: a contradiction.\\

 Hence  $\mathcal I \subset \lbrack \mathcal G, \mathcal G \rbrack$.

           Let $Y$ be a generator of $\mathcal I$, $\{Y,X'  \}$ be generators of $\lbrack \mathcal G, \mathcal G \rbrack$ and   $\{Y,X' ,Z \}$      be a basis   of $\mathcal G$. The tangent
           space of $S$, at some point $s  \in S$, is identified with $\mathcal G /  \mathcal I$ and the infinitesimal  (isotropic) action of $Y$ on this tangent space is given      in the basis
           $\{ X', Z \}$ by the matrix $ad(Y) = \left(  \begin{array}{cc}
                                                                 0  &   *\\
                                                                 0     &  0\\
                                                                 
                                                                 \end{array}  \right) $. This is because $\lbrack \mathcal G, \mathcal G \rbrack  \simeq \RR^2$ and $ad(Y)(\mathcal G) \subset \lbrack \mathcal G, \mathcal G \rbrack.$
                        Moreover, $ad(Y) \neq 0$, since the restriction to the isotropy action to the tangent space of $S$ is injective. 
                        
                        From this form of $ad(Y)$,  we see that the isotropy is unipotent with fixed direction $\RR X'$: a contradiction.\\
                        
                        We proved  that $\lbrack \mathcal{G}, \mathcal{G}  \rbrack$ is 1-dimensional. Notice that $\mathcal I \neq \lbrack \mathcal{G}, \mathcal{G}  \rbrack$. Indeed, if one assume the contrary, then the action of the isotropy on the tangent space $T_sU$ at $s \in S$ is trivial: a contradiction.

  Let $H$ be a generator of $\lbrack \mathcal G, \mathcal G \rbrack$. If $Y$ is (still) the generator of $\mathcal I$, it follows that $\lbrack Y, H \rbrack =aH$, with $a \in \RR$.
  
  Assume, by contradiction, that $a=0$. Then the image of $ad(Y)$ (which  lies in $\lbrack \mathcal G, \mathcal G \rbrack$) belongs to   the kernel of $ad(Y)$: a contradiction (since the isotropy is semi-simple).
  
  Therefore $a \neq 0$ and we can assume, by changing  the generator $Y$ of the isotropy, that $a=1$. We have  $\lbrack Y, H \rbrack =H$. 
  
  Let $X' \in \mathcal G$ such that  $\{X' ,H \}$ span the kernel of $ad(H)$. Then $\{ Y, X', H \}$ is a basis of
$\mathcal G$. We also have $\lbrack X', Y\rbrack =bH$, with $b \in \RR$. After replacing 
 $X'$ by $X' +bH$,  we can assume  $\lbrack X', Y\rbrack =0$.

 It follows that $\mathcal G$ is the Lie algebra   $\RR \oplus \mathfrak{aff}(\mathbf{R})$,  where $\mathfrak{aff}(\mathbf{R})$ is  the Lie algebra of the affine group of the real line. The Killing field $X'$ span the center, the isotropy $Y$ span the one parameter group of the homotheties  and $H$ span the one parameter group of translations in the affine group.

(ii) This comes from the fact that  $X$ is the unique vector field tangent to $S$ invariant by the isotropy.

(iii) The commuting  Killing vector fields $X'$ and $H$ are nonsingular  on $S$. This implies that, in adapted  coordinates $(x,h)$ on $S$, we get  $H= \frac{\partial}{\partial h}$ and $X=\frac{\partial}{\partial x}$. The isotropy  preserves $X$. It follows  that, in restriction to $S$,  the expression of  $Y$ is $f(h)\frac{\partial}{\partial h}$, with $f$ an analytic function vanishing at the origin. The Lie bracket relation $\lbrack Y, H \rbrack =H$ reads

$$\lbrack f(h) \frac{\partial}{\partial h},\frac{\partial}{\partial h} \rbrack = \frac{\partial}{\partial h},$$

and leads to $f(h)=-h$.  Therefore, the isotropy $Y$ is linear: $-h \frac{\partial}{\partial h}$.

(iv) Since  $H= \frac{\partial}{\partial h}$ and $X=\frac{\partial}{\partial x}$ are Killing fields, the restriction of $g$  to $S$ admits constant coefficients with respect to  the coordinates $(x,h).$ Since 
$H$ is expanded by the isotropy, it follows that $H$ is of constant $g$-norm equal to $0$. On the other hand, $X$ is of constant $g$-norm equal to one. It follows that the expression of  $g$ on $S$  is $dx^2$.

\end{proof}

\begin{lemma} \label{normal form} In adapted analytic  coordinates $(x,h,z)$,  in the neighborhood of the origin, $$g=dx^2+ dhdz+Cz^2dh^2+Dzdxdh,$$ with $C$ and $D$ real numbers.

Moreover, in these coordinates, $\frac{\partial}{\partial x}$, $\frac{\partial}{\partial h}$ and $-h\frac{\partial}{\partial h}+z\frac{\partial}{\partial z}$ are Killing fields.
\end{lemma}

\begin{proof} 
%Being central,  $X'$ is of constant norm (equal to $1$ after normalization), which implies that $X'$ is geodesic. We consider the Killing field $X$ (of constant norm equal to $1$)  and the Killing field $H$ which commutes with $X$
%and which in restriction to $S$ is  nonsingular and of constant norm $0$ (which is  $\frac{\partial}{\partial h}$ on $S$). 

%We show now that $H$ is also geodesic in restriction to $S$. 

%In restriction to $S$, the vector field  $H$ is of constant norm equal to $0$. This implies that $\nabla_HH$ is orthogonal to $X$. It follows that $\nabla_HH= \lambda H$, for some $\lambda \in \RR$.
%Since the isotropy preserves $\nabla$ and expands $H$ this implies $\lambda=0$.

Let us consider the commuting Killing vector fields $X'$ and $H$, constructed  in  Lemma~\ref{Killing}. Their restrictions to $S$ have  the expressions $H= \frac{\partial}{\partial h}$ and $X=\frac{\partial}{\partial x}$. Recall that $H$ is of constant $g$-norm equal to $0$
and $X$ is of constant $g$-norm equal to one. Point (iv) in Lemma~\ref{Killing} also  shows that $g(X,H)=0$ on $S$. Moreover, being central, $X'$ is of constant $g$-norm equal to one in  all of $U$.

We define the geodesic vector field $Z$ as follows. At each point in $S$, there exists a unique vector $Z$ (transverse to S) such that $g(Z,Z)=0, g(X,Z)=0, g(H,Z)=1$. In fact, $Z$ spans the second isotropic line
(other than that generated by $H$)   in  $X^{\bot}$. On this line $Z$ is uniquely determined by the relation $g(H,Z)=1$.  Now $Z$  uniquely extends  in the neighborhood of the origin to a geodesic vector field.

The image of $S$ through the  geodesic flow of $Z$ defines a foliation by surfaces. Each leaf is given by $exp_S(zZ)$, for some $z$, small enough. The leaf $S$ corresponds to $z=0$.

 Since $X'$ and $H$ are Killing, then $Z$ commutes with both $X'$ and $H$. Let $(x,h,z)$ be analytic coordinates in the neighborhood of the origin such that $X'=\frac{\partial}{\partial x}, H=\frac{\partial}{\partial h}, Z=\frac{\partial}{\partial z}$.

As the scalar product between the  geodesic vector field $Z$ and the  Killing vector field $X'$ is constant along the orbits of $Z$, the invariance by the commutative Killing algebra generated by
$X'$ and $H$, implies that $dxdz=0$ and $dhdz=1$. Also the coefficients of $dh^2$ and $dxdh$ depend only on $z$ (since $H$ and $X'$ are Killing fields).

We get that $g=dx^2+ dhdz+g(z)dh^2+f(z)dxdh$, with $f,g$ analytic functions which are both  constant equal to $0$ on $S$ (f(0)=g(0)=0).

We need also to write down the invariance of $g$  by the isotropy $\RR Y$.  Remember  that the Lie bracket relations in $\mathcal G$  are $\lbrack Y,X' \rbrack =0$ and $\lbrack Y,H \rbrack=H$. Since the isotropy preserves $X$, it must  preserve also the two isotropic directions of $X^{\bot}$. Moreover, since $g(H,Z)=1$, the isotropy must expand $H$ and contract  $Z$ at the same rate. 
This implies the Lie bracket relation $ad(Y)\cdot Z=\lbrack Y, Z \rbrack =-Z$.

Now, since $Y$ and $X'$ commute, the general expression for $Y$ is $u(h,z) \frac{\partial}{\partial h} + v(h,z) \frac{\partial}{\partial z}+t(h,z)\frac{\partial}{\partial x}$, with $u,v$ and $t$ analytic functions
vanishing at the origin.

The other Lie bracket relations read $ \lbrack u(h,z) \frac{\partial}{\partial h} + v(h,z) \frac{\partial}{\partial z}+t(h,z)\frac{\partial}{\partial x}, \frac{\partial}{\partial h} \rbrack =\frac{\partial}{\partial h} $ and

 $ \lbrack u(h,z) \frac{\partial}{\partial h} + v(h,z) \frac{\partial}{\partial z}+t(h,z)\frac{\partial}{\partial x}, \frac{\partial}{\partial z} \rbrack =-\frac{\partial}{\partial z} $.
 
 The first Lie bracket relation   leads to the following equations:
 $ \frac{\partial u}{\partial h}=-1$,  $\ \frac{\partial v}{\partial h}=0$,  $ \frac{\partial t}{\partial h}=0$.
 
 The second one leads to  $\frac{\partial u}{\partial z}=0$,  $\frac{\partial v}{\partial z}=1$,  $\frac{\partial t}{\partial z}=0$.
 
 We get $u(h,z)=-h$, $v(h,z)=z$, $t(h,z)=0$.

Hence, in our  coordinates,
$Y=-h\frac{\partial}{\partial h}+z\frac{\partial}{\partial z}$. The invariance of $g$ under the action of this linear vector field implies $g(e^{-t}z) e^{2t}=g(z)$ and
$f(e^{-t}z)e^t=f(z)$, for all $t \in \RR$. This implies then that $f(z)=Cz$ and $g(z)=Dz^2$, with $C, D$ real constants.
\end{proof}

\subsection{Computation of the Killing algebra}

We need to understand now if all the  metrics $$g_{C,D}=dx^2+ dhdz+Cz^2dh^2+Dzdxdh$$ constructed in Lemma~\ref{normal form} are quasihomogeneous. In other words, do the metrics in this family
 admit other Killing fields than $\frac{\partial}{\partial x}$, $\frac{\partial}{\partial h}$ and $h\frac{\partial}{\partial h}-z\frac{\partial}{\partial z}$ ?

In this section we compute the full Killing algebra of $g_{C,D}$ and prove Proposition~\ref{second}.
In particular, we obtain that the metrics $g_{C,D}=dx^2+ dhdz+Cz^2dh^2+Dzdxdh$ admit extra Killing fields and are locally homogeneous, unless $C \neq 0$ and $D =0$.

  Using the formula~\cite{kobayashi-nomizu}
  
   $$( L_{T}g_{C,D})(\indel{x_i}, \indel{x_j})=T \cdot g_{C,D}(\indel{x_i}, \indel{x_j}) + g_{C,D}(\lbrack \indel{x_i}, T \rbrack, \indel{x_j})+g_{C,D}(\indel{x_i}, \lbrack \indel{x_j}, T \rbrack)$$
   
   one gets the following PDE system for the coefficients of  a Killing field $T=\alpha \indel{x} + \beta \indel{h} + \gamma \indel{z}$.

\begin{eqnarray}
\label{primera} 0 & = & \beta_{z},\\
\label{segunda} 0 & = & \alpha_{x}+Dz  \beta_{x},\\
\label{tercera} 0 & = & \beta_{x}+Dz \beta_z+ \alpha_z,\\
\label{cuarta} 0 & = & \gamma D+Dz \alpha_x+Cz^2 \beta_x+ \gamma_x+ \alpha_h +Dz \beta_h,\\
\label{quinta} 0 & = & \beta_{h}+ Cz^2 \beta_z+Dz \alpha_z+ \gamma_z,\\
\label{sexta} 0 & = & zC \gamma +Cz^2 \beta_h + Dz \alpha_h+ \gamma_h.
\end{eqnarray}
These correspond, in the equation satisfied by $T$, to the couples

 $(\indel{x_i}, \indel{x_j})\in\left\{\left(\indel{z},\indel{z}\right),\left(\indel{x},\indel{x}\right),\left(\indel{x},\indel{z}\right), \left(\indel{x},\indel{h}\right),
\left(\indel{h},\indel{z}\right), \left(\indel{h},\indel{h}\right) \right\}$.

We prove now the final part of Proposition~\ref{second}:

\begin{proposition}

(i) If $C=0$ and $D \neq 0$, the Killing algebra of $g_{C,D}$  is solvable 
of dimension $4$ and $g$ is locally homogeneous.

(ii) If $C \neq 0$ and $D \neq 0$, then $g_{C,D}$ is a left invariant metric on $SL(2, \RR)$ and its  Killing algebra   is $\RR \oplus sl(2, \RR)$.

(iii) The Killing algebra of $g_{C,D}$  is of dimension three (and $g_{C,D}$ is quasihomogeneous) if and only if $C \neq 0$ and $D =0$.

(iv) If $C=0$ and $D=0$, $g_{C,D}$ is flat and its  Killing Lie algebra is of dimension six.
\end{proposition}

\begin{proof}  (i)  If $C=0$,  one directly checks that $\alpha =\beta =0$ and $\gamma=e^{-Dx}$ is a solution of the PDE system, meaning that
 $e^{-Dx} \indel{z}$ is an extra Killing field. In this case the Killing algebra is  of dimension $4$, generated by $\langle \indel{x}, \indel{h},h\indel{h}-z\indel{z}, e^{-Dx} \indel{z}  \rangle.$
Indeed, Lemma~\ref{dim} shows that the Killing Lie algebra cannot be  bigger,  since $g_{C,D}$ is not of constant sectional curvature, except for $C=D=0$ (see computations at point (iii)). Observe that
the Killing algebra is solvable and transitive. The Lorentz metrics $g_{0,D}$ are locally homogeneous.

(ii) Assume that   there exists an extra Killing field $T$. Then by Lemma~\ref{dim}, $g_{C,D}$ is locally homogeneous and the full  Killing Lie algebra $\mathcal G$  is of dimension four (except for $C=D=0$, for which $g_{C,D}$
 is flat and the Killing algebra is of dimension $6$), generated by $(X',Y,H,T)$.
 
 Since  the isotropy $\RR Y$ fixes $X$ and expands the direction $\RR H$ (because of the relation $\lbrack Y, H \rbrack =H$), we can choose
             as fourth generator  $T$ of $\mathcal G$,  at the origin,  a generator  of the second isotropic direction of the Lorentz plane $X^{\bot}$. Then we will have $\lbrack Y, T \rbrack =-T + aY$, for some
             constant $a \in \RR$ and we can replace $T$ with  $T -aY$ in order to get $\lbrack Y,T \rbrack =-T$.
         
        In the following, we  assume that $\lbrack Y,T \rbrack =-T$.
       
             We will first  show that       necessarily        we have
              $\lbrack H,T \rbrack =aX'- bY$, for some  $a,b \in \RR$.
     
     For this, we use the Jacobi relation  $\lbrack Y, \lbrack T, H \rbrack \rbrack=  \lbrack  \lbrack Y, T \rbrack , H \rbrack +  \lbrack T, \lbrack Y, H \rbrack \rbrack =
     \lbrack -T, H \rbrack + \lbrack T, H \rbrack =0$, to get that $ \lbrack T, H \rbrack $ commutes with  $Y$ and, consequently, lies in  $\RR Y \oplus \RR X'$.
     
     Observe also    that  $X'$ et $Y$ commute, and thus  $T$ (which is an eigenvector of  $ad(Y)$), is also  an eigenvector of $ad(X')$. This gives
     $\lbrack T, X' \rbrack= c T$, for some  $c \in \RR$.
     
     We will construct  a Killing field $T = \alpha \indel{x} + \beta \indel {h} + \gamma \indel{z}$,  such that  $c=0$ (meaning that $X'$ is central in $\mathcal G$) and $b \neq 0$ ($\mathcal G$ semi-simple).

 It follows that  $\lbrack H,T \rbrack =aX'- bY$, with  $a,b \in \RR$ and $b \neq 0$. The Lie algebra $\mathcal L$  generated by $aX'-bY,H,T$ is $sl(2,\RR)$. 
  
  We also assume that $a \neq 0$, which implies that $\mathcal L$ does not contain the isotropy and hence acts simply  transitively.   
  
  Since $T$ and $X'$ commute, the coefficients $\alpha, \beta$ and $\gamma$ do not depend on $x$.
  
  The Lie bracket relation $\lbrack H, T\rbrack =aX' -bY$ reads $\lbrack \indel{h}, T \rbrack =a \indel{x} + b(h \indel{h}-z \indel{z})$.
  
  This leads to $\alpha_h=a, \beta_h=bh, \gamma_h=-bz$.
  
  It follows that $\alpha=ah+t(z)$, $\beta=\frac{1}{2}bh^2$ and $\gamma=-bz h +s(z)$, with $t,s$ analytic functions of $z$.
  
  We check equation~\ref{quinta} which gives  $bh+Dzt'(z)+s'(z)-bh=0$,  or equivalently  $Dzt'(z) + s'(z)=0$.
  
  Equation~\ref{cuarta} leads to $\lbrack -bzh +s(z) \rbrack D +a +Dz bh$, which simplifies in 
 $Ds(z)+a=0$. 
  
  Since here  $D \neq 0$, we get  $s(z)=-\frac{a}{D}$. It follows that $Dzt'(z)=0$, and,
  since $D \neq 0$, $t'(z)=0$, and $t$ is a constant. We can assume that $t=0$ (since $X'=\indel{x}$ is a Killing field).
  
  Equation~\ref{sexta} leads to
$  -zC \lbrack bzh  + \frac{a}{D} \rbrack +Cz^2bh+Dza -bz=0$, 
which gives

$a(D - \frac{C}{D})=b$.

The extra Killling field is  $T=ah \indel{x}+ \frac{1}{2}bh^2 \indel{h} +(-bzh- \frac{a}{D}) \indel{z}$, with $a,b$ such that $a(D - \frac{C}{D})=b$.

Indeed, since $\alpha$ and $\beta$  depend only on $h$, then  $T$ also satisfies equations~\ref{primera},~\ref{segunda} and~\ref{tercera}. It is a Killing field of
$g_{C,D}$.

Consequently, $g_{C,D}$ is locally isomorphic to a left invariant metric on $SL(2, \RR)$. The isotropy is diagonally embedded in
  $\RR \oplus sl(2, \RR)$.
 This terminates the proof of point (ii).

(iii). We give here a description of the curvature of  the Lorentz metrics $g_{C,D}$.

The formulas one needs to use for the  computation are classical and  can be found in~\cite{Wolf}.  In the sequel,  the local coordinates  $(x,h,z)$ are denoted by $(x_1, x_2, x_3)$.

First one computes Christoffel coefficients $\Gamma^m_{ij}$ using the formula

$$\Gamma^m_{ij}= \frac{1}{2} \sum_k  ( \frac{\partial g_{jk}}{\partial x_i} +   \frac{\partial g_{ki}}{\partial x_j}   - \frac{\partial g_{ij}}{\partial x_k} )g^{km}.$$

Here $(g^{ij})$,  the inverse of the matrix $(g_{ij})$ is the following: $\left(  \begin{array}{ccc}
                                                                 1  &  0 & -Dz \\
                                                                 0   &  0 & 1 \\
                                                                 -Dz   &  1 & (D^2-C)z^2
                                                                 \end{array}  \right).$

For our metrics  $g_{C,D}$, only two derivates are nontrivial: $\frac{\partial g_{22}}{\partial x_3}=2zC$ and $\frac{\partial g_{12}}{\partial x_3}=D.$

We get $\Gamma^m_{12}=-\frac{1}{2}Dg^{3m}$,  $\Gamma^m_{13}=\frac{1}{2}g^{2m}\frac{\partial g_{21}}{\partial x_3}=\frac{1}{2}Dg^{2m}$, $\Gamma^m_{22}=-\frac{1}{2} g^{3m} \frac{\partial g_{22}}{\partial x_3}=-Czg^{3m}$ and $\Gamma^m_{23}=\frac{1}{2}g^{1m} \frac{\partial g_{12}}{\partial x_3}+\frac{1}{2}g^{2m}\frac{\partial g_{22}}{\partial x_3}=\frac{1}{2} Dg^{1m}+zCg^{2m}$, for $m \in \{1,2,3 \}$

A straightforward computation  leads to $\Gamma^{1}_{12}=D^2 \frac{z}{2}$, $\Gamma^2_{12}=-\frac{1}{2}D$, $\Gamma_{12}^3=-\frac{1}{2}D(D^2-C)z^2$, $\Gamma^1_{13}=\Gamma^2_{13}=0$, $\Gamma^3_{13}=\frac{1}{2}D$, $\Gamma^1_{23}= \frac{1}{2}D$, $\Gamma^2_{23}=0$, $\Gamma^3_{23}=(C-\frac{D^2}{2})z$, $\Gamma^1_{22}=CDz^2$, $\Gamma^2_{22}=-zC$, $\Gamma^3_{22}=C(C-D^2)z^3$, $\Gamma^m_{11}=0$, $\Gamma^m_{33}=0$.

Starting with Christoffel symbols we compute the curvature components using the standard formula:

$$R^s_{ijk}=\sum_l \Gamma^l_{ik} \Gamma^s_{jl} - \sum_l \Gamma^l_{jk} \Gamma^s_{il}  + \frac{\partial \Gamma^s_{ik}}{\partial x_j} -  \frac{\partial \Gamma^s_{jk}}{\partial x_i} .$$

The  straightforward computation of the components of the  curvature of the metrics $g_{C,D}$  leads to the following: 

$R^2_{121}=-\frac{1}{4}D^2, R^3_{121}=0, R^3_{131}=-\frac{1}{4}D^2, R^2_{131}=0, R^3_{132}=( \frac{3}{4}D^3-2CD)z, R^3_{232}=(-\frac{5}{4}CD^2+C^2)z^2,
R^3_{122}=0$.

This shows that for $C\neq 0$ and $D = 0$, {\it the curvature tensor  of $g_{C,D}$ vanishes exactly on the surface $z=0$.} This implies  that the Lorentz metrics $g_{C,0}$, 
with $C \neq 0$ {\it are not locally homogeneous}. They are quasihomogeneous.

Notice also that all Christoffel symbols  $\Gamma^3_{ij}$, with $i,j \in \{1,2 \}$,  vanish on  $z=0$. Consequently, $S$ is totally geodesic.

(iv) The previous computations show that the Lorentz metric $g_{C,D}$ is flat  if and only if $C=D=0$.
\end{proof}

\bibliography{references}
\bibliographystyle{alpha}

\end{document}